\documentclass{article}
\usepackage{amsfonts}
\usepackage{amsmath}

\newtheorem{theorem}{Theorem}

\newtheorem{corollary}[theorem]{Corollary}

\newtheorem{lemma}[theorem]{Lemma}

\newenvironment{proof}[1][Proof]{\noindent\textbf{#1.} }{\ \rule{0.5em}{0.5em}}
\input{tcilatex}
\begin{document}

\title{\textbf{On the binomial sums of Horadam sequence }}
\author{\textbf{Nazmiye Yilmaz and Necati Taskara}\thanks{%
e-mail: nzyilmaz@selcuk.edu.tr and\ \ \ \ \ \ \ ntaskara@selcuk.edu.tr} \\
Department of Mathematics, Science Faculty,\\
Selcuk University, 42075, Campus, Konya, Turkey}
\maketitle

\begin{abstract}
The main purpose of this paper is to establish some new properties of
Horadam numbers in terms of binomial sums. By that, we can obtain these
special numbers in a new and direct way. Moreover, some connections between
Horadam and generalized Lucas numbers are revealed to get a more strong
result.


\textit{Ams Classification: 11B39, 11B65}
\end{abstract}

\section{Introduction}

For $a,b,$ $p,q\in 
\mathbb{Z}
,$ Horadam [12] considered the sequence $W_{n}\left( {a,b\ ;\ p,q}\right) ,$
shortly $W_{n},$ which was defined by the recursive equation%
\begin{equation}
W_{n}\left( {a,b\ ;\ p,q}\right) =pW_{n-1}+qW_{n-2}~\ \ \ \ \ (n\geq 2),
\label{1}
\end{equation}%
where initial conditions are $W_{0}=a,\;W_{1}=b$ and $n\in 
\mathbb{N}
.\ $

In Equation (\ref{1}), for special choices of $a,\ b,\ p$ and $q$, the
following recurrence relations can be obtained.

\begin{itemize}
\item For $a=0,\ b=1,$ it is obtained generalized Fibonacci numbers:%
\begin{equation}
U_{n}=pU_{n-1}+qU_{n-2}.~  \label{1.1}
\end{equation}

\item For $a=2,\ b=p,$ it is obtained generalized Lucas numbers:%
\begin{equation}
V_{n}=pV_{n-1}+qV_{n-2}.  \label{1.11}
\end{equation}

\item For $a=0,\ b=1,\ p=1,\ q=1,$ it is obtained classical Fibonacci
numbers:%
\begin{equation*}
F_{n}=F_{n-1}+F_{n-2}.
\end{equation*}

\item For $a=2,\ b=1,\ p=1,\ q=1,$ it is obtained classical Lucas numbers:%
\begin{equation*}
L_{n}=L_{n-1}+L_{n-2}.
\end{equation*}

\item For $a=0,\ b=1,\ p=2,\ q=1,$ it is obtained Pell numbers:%
\begin{equation*}
P_{n}=2P_{n-1}+P_{n-2}.
\end{equation*}

\item For $a=2,\ b=2,\ p=2,\ q=1,$ it is obtained Pell-Lucas numbers:%
\begin{equation*}
Q_{n}=2Q_{n-1}+Q_{n-2}.
\end{equation*}

\item For $a=0,\ b=1,\ p=1,\ q=2,$ it is obtained Jacobsthal numbers:%
\begin{equation*}
J_{n}=J_{n-1}+2J_{n-2}.
\end{equation*}

\item For $a=2,\ b=1,\ p=1,\ q=2,$ it is obtained Jacobsthal-Lucas numbers:%
\begin{equation*}
j_{n}=j_{n-1}+2j_{n-2}.
\end{equation*}
\end{itemize}

Considering [12] (or [6]), one can clearly obtain the characteristic
equation of (\ref{1}) as the form $t^{2}-pt-q=0$ with the roots 
\begin{equation}
\alpha =\frac{p+\sqrt{p^{2}+4q}}{2}\ \ \text{and\ }\ \ \beta =\frac{p-\sqrt{%
p^{2}+4q}}{2}.  \label{1.2}
\end{equation}%
Hence the Binet formula 
\begin{equation}
W_{n}=W_{n}\left( {a,b\ ;\ p,q}\right) =A\alpha ^{n}+B\beta ^{n},
\label{1.21}
\end{equation}%
where $A=\frac{b-a\beta }{\alpha -\beta },~B=\frac{a\alpha -b}{\alpha -\beta 
},$ can be thought as a solution of the recursive equation in (\ref{1}).

The number sequences have been interested by the researchers for a long
time. Recently, there have been so many studies in the literature that
concern about subsequences of Horadam numbers such as Fibonacci, Lucas, Pell
and Jacobsthal numbers. They were widely used in many research areas as
Physics, Engineering, Architecture, Nature and Art (see [1-15]). For
example, in [7], Taskara et al. examined the properties of Lucas numbers
with binomial coefficients.

In [10], they also computed the sums of products of the terms of the Lucas
sequence $\left\{ V_{kn}\right\} .$ In addition in [11], the authors
established identities involving sums of products of binomial coefficients.

And, in [16], we obtained Horadam numbers with positive and negative indices
by using determinants of some special tridiagonal matrices.

In this study, we are mainly interested in some new properties of the
binomial sums of Horadam numbers.

\section{Main Results}

Let us first consider the following lemma which will be needed later in this
section. In fact, this lemma enables us to construct a relation between
Horadam numbers and generalized Lucas numbers by using their subscripts.

\begin{lemma}
\lbrack 10]For $n\geq 1,$ we have%
\begin{equation}
W_{ni+i}=V_{i}W_{ni}-\left( -q\right) ^{i}W_{ni-i}.  \label{03}
\end{equation}
\end{lemma}

\begin{theorem}
For $n\geq 2,$ the following equalities are hold:%
\begin{equation*}
W_{ni+i}=V_{i}^{n-1}W_{2i}-\left( -q\right)
^{i}\sum_{j=1}^{n-1}V_{i}^{n-1-j}W_{ij}.
\end{equation*}
\end{theorem}

\begin{proof}
Let us show this by induction, for $n=2,$ we can write 
\begin{equation*}
W_{3i}=V_{i}W_{2i}-\left( -q\right) ^{i}W_{i},
\end{equation*}%
which coincides with equation (\ref{03}). Now, assume that, it is true for
all positive integers $m$, i.e. 
\begin{equation}
W_{mi+i}=V_{i}^{m-1}W_{2i}-\left( -q\right)
^{i}\sum_{j=1}^{m-1}V_{i}^{m-j-1}W_{ij}.  \label{2}
\end{equation}%
Then, we need to show that above equality holds for $n=m+1,$ that is, 
\begin{equation}
W_{(m+1)i+1}=V_{i}^{m}W_{2i}-\left( -q\right)
^{i}\sum_{j=1}^{m}V_{i}^{m-j}W_{ij}.  \label{2.1}
\end{equation}%
By considering the right hand side of Equation (\ref{2.1}), we can expand
the summation as 
\begin{eqnarray*}
V_{i}^{m}W_{2i}-\left( -q\right) ^{i}\sum_{j=1}^{m}V_{i}^{m-j}W_{ij}
&=&V_{i}^{m}W_{2i}-\left( -q\right)
^{i}\sum_{j=1}^{m-1}V_{i}^{m-j}W_{ij}-\left( -q\right) ^{i}W_{mi} \\
&=&V_{i}\left( V_{i}^{m-1}W_{2i}-\left( -q\right)
^{i}\sum_{j=1}^{m-1}V_{i}^{m-j-1}W_{ij}\right) -\left( -q\right) ^{i}W_{mi}.
\end{eqnarray*}%
Then, using Equation (\ref{2}), we have%
\begin{equation*}
V_{i}^{m}W_{2i}-\left( -q\right)
^{i}\sum_{j=1}^{m}V_{i}^{m-j}W_{ij}=V_{i}W_{mi+i}-\left( -q\right) ^{i}W_{mi}
\end{equation*}%
Finally, by considering (\ref{03}), we obtain%
\begin{equation*}
V_{i}^{m}W_{2i}-\left( -q\right)
^{i}\sum_{j=1}^{m}V_{i}^{m-j}W_{ij}=W_{\left( m+1\right) i+i}
\end{equation*}%
which ends up the induction.
\end{proof}

Choosing some suitable values on $a,\ b,\ p$ and $q$, one can also obtain
the sums of the well known Fibonacci, Lucas, Pell, Jacobsthal numbers, etc.
in terms of the sum in Theorem 2.

\begin{corollary}
In Theorem 2, for special choices of $a,\ b,\ p$ and $q$, the following
result can be obtained for well-known number sequences in literature.
\end{corollary}

\begin{itemize}
\item For $a=0,\ b=1,$ it is obtained generalized Fibonacci numbers:%
\begin{equation*}
U_{ni+i}=V_{i}^{n-1}U_{2i}-\left( -q\right)
^{i}\sum_{j=1}^{n-1}V_{i}^{n-1-j}U_{ij}.
\end{equation*}

\item For $a=2,\ b=p,$ it is obtained generalized Lucas numbers:%
\begin{equation*}
V_{ni+i}=V_{i}^{n-1}V_{2i}-\left( -q\right)
^{i}\sum_{j=1}^{n-1}V_{i}^{n-1-j}V_{ij}.
\end{equation*}

\item For $a=0,\ b=1,\ p=1,\ q=1,$ it is obtained classical Fibonacci
numbers:%
\begin{equation*}
F_{ni+i}=L_{i}^{n-1}F_{2i}-\left( -1\right)
^{i}\sum_{j=1}^{n-1}L_{i}^{n-1-j}F_{ij}.
\end{equation*}

\item For $a=2,\ b=1,\ p=1,\ q=1,$ it is obtained classical Lucas numbers:%
\begin{equation*}
L_{ni+i}=L_{i}^{n-1}L_{2i}-\left( -1\right)
^{i}\sum_{j=1}^{n-1}L_{i}^{n-1-j}L_{ij}.
\end{equation*}

\item For $a=0,\ b=1,\ p=2,\ q=1,$ it is obtained Pell numbers:%
\begin{equation*}
P_{ni+i}=Q_{i}^{n-1}P_{2i}-\left( -1\right)
^{i}\sum_{j=1}^{n-1}Q_{i}^{n-1-j}P_{ij}.
\end{equation*}

\item For $a=2,\ b=2,\ p=2,\ q=1,$ it is obtained Pell-Lucas numbers:%
\begin{equation*}
Q_{ni+i}=Q_{i}^{n-1}Q_{2i}-\left( -1\right)
^{i}\sum_{j=1}^{n-1}Q_{i}^{n-1-j}Q_{ij}.
\end{equation*}

\item For $a=0,\ b=1,\ p=1,\ q=2,$ it is obtained Jacobsthal numbers:%
\begin{equation*}
J_{ni+i}=j_{i}^{n-1}J_{2i}-\left( -2\right)
^{i}\sum_{j=1}^{n-1}j_{i}^{n-1-j}J_{ij}.
\end{equation*}

\item For $a=2,\ b=1,\ p=1,\ q=2,$ it is obtained Jacobsthal-Lucas numbers:%
\begin{equation*}
j_{ni+i}=j_{i}^{n-1}j_{2i}-\left( -2\right)
^{i}\sum_{j=1}^{n-1}j_{i}^{n-1-j}j_{ij}.
\end{equation*}

\item By choosing other suitable values on $a,\ b,\ p$ and $q$, almost all
other special numbers can also be obtained in terms of the sum in Theorem 2.
\end{itemize}

\qquad Now, we will show the relation between Horadam numbers and
generalized Lucas numbers using binomial sums as follows.

\begin{theorem}
For $n\geq 2,$ the following equalities are satisfied:
\end{theorem}

\begin{equation*}
W_{ni+i}=\left\{ 
\begin{array}{c}
\sum\limits_{j=0}^{\left\lfloor \frac{n}{2}\right\rfloor }\binom{n-j}{j}%
V_{i}^{n-2j}q^{ij}W_{i}+q^{i}a\sum\limits_{j=0}^{\left\lfloor \frac{n-1}{2}%
\right\rfloor }\binom{n-j-1}{j}V_{i}^{n-2j-1}q^{ij},\ \ \ \ \ \ \ \ \ \ \ \
\ \ \ \ i\text{ is odd} \\ 
\sum\limits_{j=0}^{\left\lfloor \frac{n}{2}\right\rfloor }\binom{n-j}{j}%
\left( -1\right)
^{j}V_{i}^{n-2j}q^{ij}W_{i}-q^{i}a\sum\limits_{j=0}^{\left\lfloor \frac{n-1}{%
2}\right\rfloor }\binom{n-j-1}{j}\left( -1\right) ^{j}V_{i}^{n-2j-1}q^{ij},\
\ i\text{ is even.}%
\end{array}%
\right.
\end{equation*}

\begin{proof}
There are two cases of subscript $i$.

\textbf{Case 1}: Let be $i$ is odd. Then, by Theorem 2, we can write 
\begin{eqnarray*}
W_{ni+i} &=&V_{i}^{n-1}W_{2i}+q^{i}\sum_{j=1}^{n-1}V_{i}^{n-1-j}W_{ij} \\
&=&V_{i}^{n-1}W_{2i}+q^{i}V_{i}^{n-2}W_{i}+q^{i}V_{i}^{n-3}W_{2i}+\cdots
+q^{i}W_{\left( n-1\right) i}.
\end{eqnarray*}%
We must note that the proof should be investigated for both cases of $n$.

If $n$ is odd, then we have%
\begin{eqnarray}
W_{ni+i} &=&V_{i}^{n-2}\left( V_{i}W_{2i}+q^{i}W_{i}\right)
+q^{i}V_{i}^{n-4}\left( V_{i}W_{2i}+W_{3i}\right)  \label{04} \\
&&+\cdots +q^{i}V_{i}\left( V_{i}W_{\left( n-3\right) i}+W_{\left(
n-2\right) i}\right) +q^{i}W_{\left( n-1\right) i}.  \notag
\end{eqnarray}%
Hence, it is given the binomial summation, when the recursive substitutions
equation (\ref{04}) by using (\ref{03}), 
\begin{equation}
W_{ni+i}=\sum_{j=0}^{\frac{n-1}{2}}\dbinom{n-j}{j}%
V_{i}^{n-2j}q^{ij}W_{i}+q^{i}a\sum\limits_{j=0}^{\frac{n-1}{2}}\binom{n-j-1}{%
j}V_{i}^{n-2j-1}q^{ij}.  \label{2.11}
\end{equation}

If $n$ is even, then similar approach can be applied to obtain%
\begin{eqnarray*}
W_{ni+i} &=&V_{i}^{n-2}\left( V_{i}W_{2i}+q^{i}W_{i}\right)
+q^{i}V_{i}^{n-4}\left( V_{i}W_{2i}+W_{3i}\right) \\
&&+\cdots +q^{i}V_{i}^{0}\left( V_{i}W_{\left( n-2\right) i}+W_{\left(
n-1\right) i}\right) .
\end{eqnarray*}%
and%
\begin{equation}
W_{ni+i}=\sum_{j=0}^{\frac{n}{2}}\dbinom{n-j}{j}%
V_{i}^{n-2j}q^{ij}W_{i}+q^{i}a\sum\limits_{j=0}^{\frac{n-2}{2}}\binom{n-j-1}{%
j}V_{i}^{n-2j-1}q^{ij}.  \label{2.12}
\end{equation}%
For the final step, we combine (\ref{2.11}) and (\ref{2.12}) to see the
equality%
\begin{equation*}
W_{ni+i}=\sum_{j=0}^{\left\lfloor \frac{n}{2}\right\rfloor }\dbinom{n-j}{j}%
V_{i}^{n-2j}q^{ij}W_{i}+q^{i}a\sum\limits_{j=0}^{\left\lfloor \frac{n-1}{2}%
\right\rfloor }\binom{n-j-1}{j}V_{i}^{n-2j-1}q^{ij},
\end{equation*}%
as required. Now, for the next case, consider

\textbf{Case 2}: Let be $i$ is even. Then, by Theorem 2, we know%
\begin{eqnarray*}
W_{ni+i} &=&V_{i}^{n-1}W_{2i}-q^{i}\sum_{j=1}^{n-1}V_{i}^{n-1-j}W_{ij} \\
&=&V_{i}^{n-1}W_{2i}-q^{i}V_{i}^{n-2}W_{i}-q^{i}V_{i}^{n-3}W_{2i}-\cdots
-q^{i}W_{\left( n-1\right) i}.
\end{eqnarray*}%
and therefore, we write%
\begin{equation}
W_{ni+i}=\sum_{j=0}^{\frac{n-1}{2}}\dbinom{n-j}{j}\left( -1\right)
^{j}V_{i}^{n-2j}q^{ij}W_{i}-q^{i}a\sum\limits_{j=0}^{\frac{n-1}{2}}\binom{%
n-j-1}{j}\left( -1\right) ^{j}V_{i}^{n-2j-1}q^{ij}  \label{2.13}
\end{equation}%
if $n$ is odd. And we get%
\begin{equation}
W_{ni+i}=\sum_{j=0}^{\frac{n}{2}}\dbinom{n-j}{j}\left( -1\right)
^{j}V_{i}^{n-2j}q^{ij}W_{i}-q^{i}a\sum\limits_{j=0}^{\frac{n-2}{2}}\binom{%
n-j-1}{j}\left( -1\right) ^{j}V_{i}^{n-2j-1}q^{ij}  \label{2.14}
\end{equation}%
if $n$ is even. Thus, by combining (\ref{2.13}) and (\ref{2.14}), we obtain%
\begin{equation*}
W_{ni+i}=\sum_{j=0}^{\left\lfloor \frac{n}{2}\right\rfloor }\dbinom{n-j}{j}%
\left( -1\right)
^{j}V_{i}^{n-2j}q^{ij}W_{i}-q^{i}a\sum\limits_{j=0}^{\left\lfloor \frac{n-1}{%
2}\right\rfloor }\binom{n-j-1}{j}\left( -1\right) ^{j}V_{i}^{n-2j-1}q^{ij}.
\end{equation*}%
Hence the result follows.
\end{proof}

Choosing some suitable values on $i,\ a,\ b,\ p$ and $q$, one can also
obtain the binomial sums of the well known Fibonacci, Lucas, Pell,
Jacobsthal numbers, etc. in terms of binomial sums in Theorem 4.

\begin{corollary}
In Theorem 4, for special choices of $i,\ a,\ b,\ p,\ q$, the following
result can be obtained.
\end{corollary}

\begin{itemize}
\item For $i=1,$

\begin{itemize}
\item[*] For $a=0$ and $b,p,q=1$, Fibonacci number 
\begin{equation*}
F_{n+1}=\sum\limits_{j=0}^{\left\lfloor \frac{n}{2}\right\rfloor }\binom{n-j%
}{j},
\end{equation*}

\item[*] For $a=2$ and $b,p,q=1$, Lucas number%
\begin{equation*}
L_{n+1}=\sum\limits_{j=0}^{\left\lfloor \frac{n}{2}\right\rfloor }\binom{n-j%
}{j}+2\sum\limits_{j=0}^{\left\lfloor \frac{n-1}{2}\right\rfloor }\binom{%
n-j-1}{j}.
\end{equation*}

\item[*] For $a=0$,$\ b=1,\ p=2$ and $q=1$, Pell number 
\begin{equation*}
P_{n+1}=\sum\limits_{j=0}^{\left\lfloor \frac{n}{2}\right\rfloor }\binom{n-j%
}{j}2^{n-2j}.
\end{equation*}

\item[*] For $a=0$,$\ b=1,\ p=1$ and $q=2$, Jacobsthal number%
\begin{equation*}
J_{n+1}=\sum\limits_{j=0}^{\left\lfloor \frac{n}{2}\right\rfloor }\binom{n-j%
}{j}2^{j}.
\end{equation*}
\end{itemize}

\item For $i=2,$

\begin{itemize}
\item[*] For $a=0$ and $b,p,q=1$, Fibonacci number 
\begin{equation*}
F_{2n+2}=\sum\limits_{j=0}^{\left\lfloor \frac{n}{2}\right\rfloor }\binom{n-j%
}{j}\left( -1\right) ^{j}3^{n-2j}.
\end{equation*}

\item[*] For $a=2$ and $b,p,q=1$, Lucas number%
\begin{equation*}
L_{2n+2}=\sum\limits_{j=0}^{\left\lfloor \frac{n}{2}\right\rfloor }\binom{n-j%
}{j}\left( -1\right) ^{j}3^{n+1-2j}-2\sum\limits_{j=0}^{\left\lfloor \frac{%
n-1}{2}\right\rfloor }\binom{n-j-1}{j}\left( -1\right) ^{j}3^{n-1-2j}.
\end{equation*}

\item[*] For $a=0$,$\ b=1,\ p=2$ and $q=1$, Pell number 
\begin{equation*}
P_{2n+2}=2\sum\limits_{j=0}^{\left\lfloor \frac{n}{2}\right\rfloor }\binom{%
n-j}{j}\left( -1\right) ^{j}6^{n-2j}.
\end{equation*}

\item[*] For $a=0$,$\ b=1,\ p=1\ $and $q=2$, Jacobsthal number 
\begin{equation*}
J_{2n+2}=\sum\limits_{j=0}^{\left\lfloor \frac{n}{2}\right\rfloor }\binom{n-j%
}{j}\left( -1\right) ^{j}2^{j}.
\end{equation*}
\end{itemize}

\item By choosing other suitable values on $i,\ a,\ b,\ p$ and $q$, almost
all other special numbers can also be obtained in terms of the binomial sum
in Theorem 4.
\end{itemize}

\end{document}